\documentclass[10pt]{amsart}
\usepackage{indentfirst,enumerate,cite,amssymb,amsfonts,amsmath,amsthm,mathrsfs,dsfont}
\usepackage{color}
\usepackage{multicol}
\usepackage[colorlinks=true,linkcolor=blue,citecolor=blue]{hyperref}
\usepackage{enumerate}
\usepackage{geometry}
\numberwithin{equation}{section}

\theoremstyle{plain}
\newtheorem{thm}{Theorem}[section]
\newtheorem{prop}[thm]{Proposition}

\newtheorem{lemma}[thm]{Lemma}

\newtheorem{example}[thm]{Example}

\newcommand{\Rep}{\mbox{Rep}}

\newcommand{\N}{{\mathbb{N}}}

\newcommand{\R}{\mathbb{R}}
\newcommand{\Z}{\mathbb{Z}}
\newcommand{\C}{\mathbb{C}}

\newcommand{\jp}[1]{{\left\langle{#1}\right\rangle}}

\newlength{\dhatheight}

\begin{document}
	\title[GH for a class of evolution equations on compact Lie groups]	{Global hypoellipticity for a class of complex-valued evolution equations on compact Lie groups}


\author[Wagner A. A. de Moraes]{Wagner A. A. de Moraes}
\address{
Universidade Federal do Paran\'{a}, 
Departamento de Matem\'{a}tica,
C.P.19096, CEP 81531-990, Curitiba, Brazil
}
\email{wagnermoraes@ufpr.br}





\begin{abstract}
We present necessary and sufficient conditions to have global hypoellipticity for a class of complex-valued coefficient first order evolution equations defined on $\mathbb{T}^1 \times G$, where $G$ is a compact Lie group. First, we show that the global hypoellipticity of the constant coefficient operator related to this operator is a necessary condition, but not a sufficient condition. Under certain hypothesis, we show that the global hypoellipticity of this class of operator is completely  characterized by Nirenberg-Treves' condition $(\mathcal{P})$.
\end{abstract}
\raggedbottom
\maketitle
\raggedbottom
\section{Introduction}

In this note, we study the global hypoellipticity of a class of first order evolution equations on $\mathbb{T}^1 \times G$, where $G$ is a compact Lie group. Precisely, 
we consider the operator $L: \mathcal{D}'(\mathbb{T}^1\times G) \to \mathcal{D}'(\mathbb{T}^1\times G)$ defined 
	$$
	L:=\partial_t+c(t)X+q,
	$$ 
	where $X$ is a vector field on $G$, $c \in C^\infty  (\mathbb{T}^1)$, $c(t)=a(t)+i b(t)$, and $q\in \C$ and we are interested to establish conditions on $c$ and $q$ to ensure that $u\in C^\infty(\mathbb{T}^1\times G)$ whenever $Lu \in C^{\infty}(\mathbb{T}^1\times G)$.
	
	We can summarize the results that we present in this work as follows:
	\begin{thm}\label{thmbsum}
		Let $G$ be a compact Lie group and consider the operator $L:\mathcal{D}'(\mathbb{T}^1\times G) \to \mathcal{D}'(\mathbb{T}^1\times G)$ defined as
		$$
		L:=\partial_t+c(t)X+q,
		$$ 
		where $X\in\mathfrak{g}$, $c \in C^\infty(\mathbb{T}^1)$, $c(t)=a(t)+i b(t)$, and $q\in \C$. Assume that 
		\begin{enumerate}[a)]
			\item $b \not \equiv 0$;
			\item $L_{0}:=\partial_t+c_0X+q$ is globally hypoelliptic, where $c_0$ is the average of $c$.
		\end{enumerate}
		Then $L$ is globally hypoelliptic if and only if $b$ does not change sign.
	\end{thm}

The case where $b\equiv 0$ was completely characterized in \cite{KMR19b} and it was obtained that $L$ is globally hypoelliptic if and only if $L_0$ is globally hypoelliptic, where $L_0=\partial_t+a_0 X+q$, with $a_0$ the average of $a$ on $\mathbb{T}^1$.

In \cite{GW72}, Greenfield and Wallach considered the case where $G=\mathbb{T}^1$, $c$ is a constant, and $q=0$, in which they obtained that $L$ is globally hypoelliptic if, and only if, either $\operatorname{Im}(c)\neq 0$ or $c$ is a non-Liouville irrational. When $c$ is not constant and $q=0$, J. Hounie showed in \cite{Hou79} that in the case where $b\equiv 0$, the operator $L$ is globally hypoelliptic if, and only if, $L_0=\partial_t+a_0\partial_x$ is globally hypoelliptic, where $a_0$ is the average of $a$. On the other hand, when $b\not\equiv 0$ the global hypoellipticity of $L$ is characterized by Nierember--Treves' condition $(\mathcal{P})$, which in this case translates to the property of the function $b$ does not change sign.

In \cite{GW73a}, Greenfield and Wallach proposed the following conjecture: if a closed smooth orientable manifold admits a globally hypoelliptic vector field, then this manifold is $C^\infty-$diffeomorphic to a torus, and this vector field is $C^\infty-$conjugated to a constant vector field whose coefficients satisfy a Diophantine condition. The validity of this conjecture was proved for some particular cases, including in the same article the proof for the case where the manifold is a compact Lie group. In light of this result, it is natural to assume that $q \neq 0$ when considering $G$ as a non-commutative compact Lie group, as for example in \cite{KMP21} and \cite{KMR19b}.

Still in the toroidal case, in \cite{Ber94} Bergamasco studied perturbations of globally hypoelliptic vector fields by zeroth order terms. One of the results obtained was that in the case where $b$ does not change sign, then for any $q$, the operator remains globally hypoelliptic. It is worth noting here that this phenomenon does not extend when considering $G$ as any compact Lie group. As stated in Theorem 1.1, the fact that $b$ does not change sign is not sufficient for the global hypoellipticity of $L$, and will be illustrated in Example \ref{exe}.

The paper is organized as follow. In Section \ref{Sec-2}, we present some preliminaries results and the notation used throughout the paper is fixed. In Section \ref{Sec-3} we introduce the operator in question, accompanied by a discussion of the previously results related to it. In Section \ref{Sec-4} we give necessary conditions and in Section \ref{Sec-5} we present the sufficient conditions of the Theorem \ref{thmbsum}.

\section{Preliminaries Results}\label{Sec-2}

In this section, we recall most of the notations and preliminary results necessary for the development of this study. The references \cite{FR16}  and \cite{livropseudo} contain a very detailed discussion of these concepts as well as proofs of all of the results presented here.

Let $G$ be a compact Lie group of dimension $\dim G=d$ and let $\operatorname{Rep}(G)$  be the set of continuous irreducible unitary representations of $G$. Since $G$ is compact, every continuous irreducible unitary representation $\eta$ is finite dimensional and it can be viewed as a matrix-valued function $\eta: G \to \mathbb{C}^{d_\eta\times d_\eta}$, where $d_\eta = \dim \eta$. We say that $\eta \sim \psi$ if there exists an unitary matrix $A\in C^{d_\eta \times d_\eta}$ such that $A\eta(x) =\psi(x)A$, for all $x\in G$. We will denote by $\widehat{G}$ the quotient of $\operatorname{Rep}(G)$ by this equivalence relation.

For $f \in L^1(G)$ the group Fourier transform of $f$ at $\eta \in \operatorname{Rep}(G)$ is
\begin{equation*}
	\widehat{f}(\eta)=\int_G f(x) \eta(x)^* \, dx,
\end{equation*}
where $dx$ is the normalized Haar measure on $G$.
By the Peter-Weyl theorem, we have that 
\begin{equation*}
	\mathcal{B} := \left\{\sqrt{d_\eta} \, \eta_{rs} \,; \ \eta=(\eta_{rs})_{r,s=1}^{d_\eta}, [\eta] \in \widehat{G} \right\},
\end{equation*}
is an orthonormal basis for $L^2(G)$, where we pick only one matrix unitary representation in each class of equivalence, and we may write
\begin{equation*}
	f(x)=\sum_{[\eta]\in \widehat{G}}d_\eta \operatorname{Tr}(\eta(x)\widehat{f}(\eta)).
\end{equation*}
Moreover, the Plancherel formula holds:
\begin{equation*}
	\label{plancherel} \|f\|_{L^{2}(G)}=\left(\sum_{[\eta] \in \widehat{G}}  d_\eta \ 
	\|\widehat{f}(\eta)\|_{{\mathtt{HS}}}^{2}\right)^{1/2}=:
	\|\widehat{f}\|_{\ell^{2}(\widehat{G})},
\end{equation*}
where 
\begin{equation*} \|\widehat{f}(\eta)\|_{{\mathtt{HS}}}^{2}=\operatorname{Tr}(\widehat{f}(\eta)\widehat{f}(\eta)^{*})=\sum_{r,s=1}^{d_\eta}  \bigr|\widehat{f}(\eta)_{rs}\bigr|^2.
\end{equation*}

The group Fourier transform of $u\in \mathcal{D}'(G)$ at a matrix unitary representation $\eta$ is the matrix $\widehat{u}(\eta) \in \mathbb{C}^{d_\eta \times d_\eta}$, whose components are given by
$$
\widehat{u}(\eta)_{rs} = \jp{u,\overline{\eta_{sr}}},
$$
where $\jp{\,\cdot\,,\,\cdot\,}$ denotes the distributional duality.

Let $\mathcal{L}_G$ be the Laplace-Beltrami operator of $G$. For each $[\eta] \in \widehat{G}$, its matrix elements are eigenfunctions of $\mathcal{L}_G$ correspondent to the same eigenvalue that we will denote by $-\nu_{[\eta]}$, where $\nu_{[\eta]} \geq 0$. Thus
\begin{equation*}\label{laplacian}
	-\mathcal{L}_G \eta_{rs}(x) = \nu_{[\eta]}\eta_{rs}(x), \quad \textrm{for all } 1 \leq r,s \leq d_\eta,
\end{equation*}
and we will denote by
$$
\jp \eta := \left(1+\nu_{[\eta]}\right)^{1/2}
$$
the eigenvalues of $(I-\mathcal{L}_G)^{1/2}.$ 

For $x\in G$, $X\in \mathfrak{g}$ and $f\in C^\infty(G)$, define 
$$
L_Xf(x):=\frac{d}{dt} f(x\exp(tX))\bigg|_{t=0}, 
$$
 and when there is no possibility of ambiguous meaning, we write only $X$ instead of $L_X$. 

Let  $P: C^{\infty}(G) \to C^{\infty}(G)$ be a continuous linear operator. The  symbol of the operator $P$ in $x\in G$ and $\eta \in \mbox{{Rep}}(G)$, $\eta=(\eta_{rs})_{r,s=1}^{d_\eta}$ is
$$
\sigma_P(x,\eta) := \eta(x)^*(P\eta)(x) \in \mathbb{C}^{d_\eta \times d_\eta},
$$
where $(P\eta)(x)_{rs}:= (P\eta_{rs})(x)$, for all $1\leq r,s \leq d_\eta$, and we have
$$
Pf(x) = \sum_{[\eta] \in \widehat{G}} \dim (\eta) \operatorname{Tr} \left(\eta(x)\sigma_P(x,\eta)\widehat{f}(\eta)\right),
$$
for all $f \in C^\infty(G)$ and $x\in G$.

When $P: C^\infty(G) \to C^\infty(G)$ is a continuous linear left-invariant operator we have that $\sigma_P$ is independent of $x\in G$ and
$$
\widehat{Pf}(\eta) = \sigma_P(\eta)\widehat{f}(\eta),
$$
for all $f \in C^\infty(G)$ and $[\eta] \in \widehat{G}$ and, by duality, this remains true for all $f \in \mathcal{D}'(G)$.

	We have that $X$ is a left-invariant operator on $G$ and $iX$ is a symmetric operator on $L^2(G)$. Thus, for each $[\eta] \in \widehat{G}$ we can choose a representative $\eta \in \Rep({G})$ such that
\begin{equation}\label{choose}
	\sigma_{X}(\eta)_{rs} =i\mu_r(\eta) \delta_{rs}, \quad 1 \leq r,s \leq d_\eta,
\end{equation}
with $\mu_r(\eta)\in\R$ for all $[\eta] \in\widehat{G}$, and we have
\begin{equation}\label{controleigen}
	|\mu_r(\eta)| \leq C \jp{\eta},
\end{equation}
for all $[\eta] \in \widehat{G}$, $1 \leq r \leq d_\eta$, where $C>0$ is a constant which depends only on $X$. Moreover, as $iX$ is not bounded from $L_s^2(G)$ to $L_{s-\frac{1}{2}}^2(G)$, one can verify that there must exist a sequence $([\eta_j])_{j\in\mathbb{N}}$ in $\widehat{G}$ and $C>0$ such that
\begin{equation}\label{growth}
	\mu_{r}(\eta_j)\geq C\langle \eta_j\rangle^{\frac{1}{2}},
\end{equation}
for some $1\leq r \leq d_{\eta_j}$ and every $j\in\N$. We may assume without loss of generality that $r=d_{\eta_j}$, for all $j\in \mathbb{N}$. For a detailed discussion about these properties see  \cite{CKKM24}, \cite{KMR19b} and \cite{livropseudo}.

Given $f \in L^1(\mathbb{T}^1 \times G)$ and  $\eta \in {\operatorname{Rep}}(G)$, the partial Fourier coefficient of $f$ with respect to the second variable at $\eta$ is defined by 
$$
\widehat{f}(t, \eta) := \int_{G} f(t,x)\, \eta(x)^* \, dx \in \mathbb{C}^{d_\eta \times d_\eta}, \quad t \in \mathbb{T}^1,
$$
with components
$$
\widehat{f}(t, \eta)_{rs} = \int_{G} f(t,x)\, \overline{\eta(x)_{sr}} \, dx, \quad 1 \leq r,s\leq d_\eta.
$$

Given $u \in \mathcal{D}'(\mathbb{T}^1\times G)$, $\eta \in {\operatorname{Rep}}(G)$ and $1\leq r,s \leq d_\eta$. The $rs$-component of  the partial Fourier coefficient of $u$ with respect to the second variable at $\eta$ is the linear functional defined by
$$
\begin{array}{rccl}
	\widehat{u}(\: \cdot\:,\eta )_{rs}: & C^\infty(\mathbb{T}^1) & \longrightarrow & \mathbb{C} \\
	& \psi & \longmapsto & \jp{\widehat{u}(\: \cdot \:,\eta)_{rs},\psi} := \jp{u,\psi\times\overline{\eta_{sr}}},
\end{array}
$$
where $\psi\times\overline{\eta_{sr}}(t,x):=\psi(t)\overline{\eta_{sr}(x)}$.

We have the following characterization of $C^\infty(\mathbb{T}^1\times G)$ and $\mathcal{D}'(\mathbb{T}^1\times G)$ through their partial Fourier coefficients presentend on \cite{KMR19}.
\begin{thm}[Theorem 3.5 of \cite{KMR19}]\label{caracsmooth}
	Let $G$ be a compact Lie group, and let $\{\widehat{f}(\:\cdot \:, \eta)_{rs} \}$ be a sequence of functions on $\mathbb{T}^1$. Define
	$$
	f(t,x) := \sum_{[\eta]\in \widehat{G}} d_\eta \sum_{r,s=1}^{d_\eta} \widehat{f}(t, \eta)_{rs} \eta_{sr}(x).
	$$ 
	Then $f \in C^\infty(\mathbb{T}^1 \times G)$ if and only if $\widehat{f}(\: \cdot\:,\eta)_{rs} \in C^\infty(\mathbb{T}^1)$, for all $[\eta]\in\widehat{G}$, $1 \leq r,s \leq d_\eta$ and for every $\beta \in \N_0$ and $\ell >0$ there exist $C_{\beta \ell} > 0$ such that
	\begin{equation*}\label{carac}
		\bigl|\partial_t^\beta \widehat{f}(t,\eta)_{rs}\bigl| \leq C_{\beta\ell} \langle \eta \rangle ^{-\ell}, \quad \forall t\in \mathbb{T}^1, \ [\eta]\in \widehat{G}, \ 1\leq r,s\leq d_\eta.
	\end{equation*}
\end{thm}
\begin{thm}[Theorem 3.6 of \cite{KMR19}]\label{caracdist}
	Let $G$ be a compact Lie groups and $\bigl\{\widehat{u}(\: \cdot \:,\eta)_{rs} \bigr\}$ be a sequence of distributions on $\mathbb{T}^1$.  Define \begin{equation*}
		u=\sum_{[\eta]\in \widehat{G}} d_\eta \sum_{r,s=1}^{d_\eta} \widehat{u}(\:\cdot\:,\eta)_{rs}{\eta_{sr}}
	\end{equation*} 
	Then $u \in \mathcal{D}'(\mathbb{T}^1 \times G)$ if and only if there exist $K \in \N$ and $C>0$ such that
	\begin{equation*}\label{pk}
		\bigl| \jp{\widehat{u}(\cdot,\eta)_{rs},\varphi}\bigr| \leq C \, p_K(\varphi) \langle \eta \rangle ^{K}, 
	\end{equation*} 
	for all $\varphi \in C^{\infty}(\mathbb{T}^1)$ and $[\eta]\in \widehat{G}$, where $p_K(\varphi) := \sum\limits_{\beta \leq K} \|\partial_t^\beta \varphi\|_{L^\infty(\mathbb{T}^1)}.$
\end{thm}
	\section{A class of first order evolution equations}\label{Sec-3}

	Let $G$ be a compact Lie group and consider the operator $L: \mathcal{D}'(\mathbb{T}^1\times G) \to \mathcal{D}'(\mathbb{T}^1\times G)$ defined 
	\begin{equation}\label{operador}
			L:=\partial_t+c(t)X+q,
	\end{equation}
	where $X$ is a vector field on $G$, $c(t)=a(t)+i b(t)$, with $a,b \in C^\infty(\mathbb{T}^1;\mathbb{R})$ and $q\in \C$.
	
	 The global hypoellipticity of \eqref{operador} in the case where either $c$ is a constant function or $\operatorname{Im}(c)=b\equiv 0$ was completely characterized in \cite{KMR19b} in a more general context that can be adapted to our setting as follow.
	 \begin{thm}\label{GH-KMR}[Theorems 3.3 and 5.1 of \cite{KMR19b}]
	 	Let $G$ be a compact Lie group, $X \in \mathfrak{g}$ and $c,q \in \mathbb{C}$. The operator $L=\partial_t+c X+q$ is globally hypoelliptic  if and only if the following conditions are satisfied: 
	 \begin{enumerate}[1.]
	 	\item The set
	 	$$
	 	\mathcal{N}=\{(k, [\eta]) \in  \mathbb{Z} \times \widehat{G}; \ k+c \mu_r(\eta)-iq = 0, \ \mbox{for some } 1 \leq r \leq d_\eta  \}
	 	$$
	 	is finite.
	 	\item $\exists C, \, M>0$ such that
	 	\begin{equation*}\label{hypothesis}
	 		|k+c \mu_r(\eta)-iq|\geq C (|k| +\langle \eta \rangle )^{-M}, 
	 	\end{equation*}
	 	for all  $k \in \mathbb{Z}, \ [\eta] \in \widehat{G}, \  1 \leq r \leq d_\eta,$ whenever $k+c \mu_r(\eta)-iq \neq 0$.
	 \end{enumerate}
 \end{thm}
	 \begin{thm}\label{GH-KMR2}[Propositions 6.6 and 6.12 of \cite{KMR19b}]
	 Let $G$ be a compact Lie group, $X \in \mathfrak{g}$, $q \in \mathbb{C}$ and $a(t)$ be a smooth real-valued function on $\mathbb{T}^1$. The operator $L=\partial_t + a(t)X+q$ is globally hypoelliptic if and only if the operator $L_{0}=\partial_t + a_0X+q$ is globally hypoellitpic, where $a_0 = \int_{\mathbb{T}^1} a(s) \mathrm{d} s$.
\end{thm}
  Hence, we will assume throughout the paper that $b(t)$ is not a constant function. 

Notice that in Theorem  \ref{GH-KMR2} the global hypoellipticity of $L$  when $b \equiv 0$ is strictly related to the global hypoellipticity of the constant coefficient operator
	$
	L_{0}=\partial_t+a_0X+q
	$
	because these two operator are conjugated by the automorphism 
	\begin{equation}\label{autom}
		\Psi_a  u(t,x) := \sum_{[\eta]\in \widehat{G}}d_\eta \sum_{r,s = 1}^{d_\eta} e^{i\mu_r(\eta) A(t)}\widehat{u}(t,\eta)_{rs}\,{\eta_{sr}(x)},
	\end{equation} 
	where $A(t)$ is a primitive for $a(t)$, that is, $\Psi_a \circ L = L_{0} \circ \Psi_a$ (for more details, see Section 6 of \cite{KMR19b}). By the same approach, in the case where $b$ is not a constant function, we could define an operator $\Psi_c$ that conjugates the operators $L$ and $L_{0}$, but we would not have the growth control of the term $e^{-\mu_r(\eta)B(t)}$ that would appear in the definition of $\Psi_c$, which implies that $\Psi_c$ is not an automorphism in the space of smooth functions and therefore the global hypoellipticity of $L$ is not equivalent to the global hypoellipticity of $L_{0}$.  Although these properties are not equivalent, we will prove in Proposition \ref{implicationGH} that the global hypoellipticity of $L_{0}$ is a necessary condition for the global hypoellipticity of $L$ and we give an example where this is not a sufficient condition. 
	\newpage 
	First, observe that by the automorphism $\Psi_a$ given in \eqref{autom} we may assume that $a(t)$ is a constant function. Henceforth we will study the operator,
	\begin{equation}\label{op-L}
			L=\partial_t+(a_0+ib(t))X+q,
	\end{equation}
	where $b$ is a real-valued smooth function on $\mathbb{T}^1$, $X$ is a vector field on $G$ and $q \in \mathbb{C}$.

	Consider the equation $Lu=f \in C^\infty(\mathbb{T}^1\times G)$. Taking the partial Fourier coefficient with respect to the second variable we obtain
	$$
	\widehat{f}(t,\eta)=\widehat{Lu}(t,\eta) = \partial_t \widehat{u}(t,\eta) + c(t)\sigma_X(\eta)\widehat{u}(t,\eta)+q\widehat{u}(t,\eta), \quad [\eta] \in \widehat{G},
	$$
	where the representation $\eta$ satisfies \eqref{choose}. Hence, for each $[\eta] \in \widehat{G}$ we obtain the following system of equations
	\begin{equation*}\label{fouriereq}
		\widehat{f}(t,\eta)_{rs}=\widehat{Lu}(t,\eta)_{rs} = \partial_t \widehat{u}(t,\eta)_{rs} + i(\mu_r(\eta)c(t)-iq)\widehat{u}(t,\eta)_{rs},
	\end{equation*}
	with $1\leq r,s\leq d_\eta$.
	
	Let 
	$$
	C(t) = \int_{0}^{t}c(\tau)\mathrm{d}\tau - c_0t, \quad \text{where} \quad  c_0 = \frac{1}{2\pi} \int_{0}^{2\pi} c(\tau)\mathrm{d}\tau.
	$$
	Multiplying by $e^{i\mu_r(\eta)C(t)}$, we obtain
	$$
	\partial_t \widehat{u}(t,\eta)_{rs}e^{i\mu_r(\eta)C(t)} + i(\mu_r(\eta)c(t)-iq)\widehat{u}(t,\eta)_{rs}e^{i\mu_r(\eta)C(t)} = \widehat{f}(t,\eta)_{rs}e^{i\mu_r(\eta)C(t)}
	$$
	Then
	$$
	\partial_t \bigl[ \widehat{u}(t,\eta)_{rs}e^{i\mu_r(\eta)C(t)} \bigl] + i(\mu_r(\eta)c_0-iq)\widehat{u}(t,\eta)_{rs}e^{i\mu_r(\eta)C(t)} = \widehat{f}(t,\eta)_{rs}e^{i\mu_r(\eta)C(t)},
	$$
	that is, for each $\eta \in \widehat{G}$ and $1\leq r,s \leq d_\eta$, we have that $\widehat{u}(t,\eta)_{rs}e^{i\mu_r(\eta)C(t)}$ is a solution of
	\begin{equation}\label{systemode}
		\partial_t v(t,\eta)_{rs} + i(\mu_r(\eta)c_0-iq) v(t,\eta)_{rs} = g(t,\eta)_{rs},
	\end{equation}
	where $g(t,\eta)_{rs}=\widehat{f}(t,\eta)_{rs}e^{i\mu_r(\eta)C(t)}$. The solution of this differential equation on $\mathbb{T}^1$ is given by: 
	\begin{lemma}\label{bsolutionconstant}
			Let $\lambda \in \C$ and consider the equation
		\begin{equation}\label{bequationconstant}
			\frac{d}{dt}y(t)+\lambda y(t)=h(t),
		\end{equation}
		where $h \in C^\infty(\mathbb{T}^1)$.
		
		If $\lambda \notin i\Z$ then the equation \eqref{bequationconstant} has a unique solution that can be expressed by
		\begin{equation*}\label{bsolutionminus}
			y(t)=\frac{1}{1-e^{-2\pi\lambda}} \int_0^{2\pi} e^{-\lambda s} h(t-s)\, \mathrm{d}s,
		\end{equation*}
		or equivalently,
		\begin{equation*}\label{bsolutionplus}
			y(t)=\frac{1}{e^{2\pi\lambda}-1} \int_0^{2\pi} e^{\lambda r} h(t+r)\, dr.
		\end{equation*}
		
		If $\lambda \in i\Z$ and $\int_0^{2\pi} e^{\lambda s}f(s)\, \mathrm{d}s=0$ then we have that
		\begin{equation*}\label{bsolutionzero}
			y(t)=e^{-\lambda t}\int_0^t e^{\lambda s}h(s) \, \mathrm{d}s
		\end{equation*}
		is a solution of the equation \eqref{bequationconstant}.
	\end{lemma}
	It follows from Lemma \ref{bsolutionconstant} that \eqref{systemode} has a unique solution given by
	\begin{equation*}
		v(t,\eta)_{rs}= \frac{1}{1-e^{-2\pi i (\mu_r(\eta)c_0-iq)}} \int_{0}^{2\pi} e^{-i(\mu_r(\eta)c_0-iq) \tau} g(t-\tau,\eta)_{rs}\mathrm{d}\tau,
	\end{equation*}
	whenever $\mu_r(\eta)c_0-iq \notin \Z$, or equivalently by
	\begin{equation*}
		v(t,\eta)_{rs}= \frac{1}{e^{2\pi i( \mu_r(\eta)c_0-iq)}-1} \int_{0}^{2\pi} e^{i(\mu_r(\eta)c_0-iq) \tau} g(t+\tau,\eta)_{rs}\mathrm{d}\tau.
	\end{equation*} 
	Therefore, we obtain
	\begin{equation}\label{solution2}
		\widehat{u}(t,\eta)_{rs}=\frac{1}{1-e^{-2\pi i( \mu_r(\eta)c_0-iq)}} \int_{0}^{2\pi} e^{-q\tau}e^{-i\mu_r(\eta)(c_0\tau -C(t-\tau)+C(t))} \widehat{f}(t-\tau,\eta)_{rs}\mathrm{d}\tau,
	\end{equation}
	or equivalently,
	\begin{equation}\label{solutionplus}
		\widehat{u}(t,\eta)_{rs}=\frac{1}{e^{2\pi i (\mu_r(\eta)c_0-iq)}-1} \int_{0}^{2\pi} e^{q\tau} e^{-i\mu_r(\eta)(-c_0\tau -C(t+\tau)+C(t))} \widehat{f}(t+\tau,\eta)_{rs}\mathrm{d}\tau.
	\end{equation}
	
To study the global properties of the operator $L$ we have to control the behavior of the numerical sequence 
	that precedes the integral in the expression above. Thus, we will use the following technical 
	lemma which proof is similar to the toroidal case studied in \cite{AGKM18}.
	
	\begin{lemma}\label{expestimate}
		Are equivalent:
		\begin{description}
			\item[1.] There exist $C,M>0$ such that
			\begin{equation*}
				|k+c_0\mu_r(\eta)-iq| \geq C(|k|+\jp{\eta})^{-M},
			\end{equation*}
			for all $k\in \Z$, $[\eta]\in \widehat{G}$, $1 \leq r \leq d_\eta$, whenever $k+c_0\mu_r(\eta)-iq\neq0$.
			\item[2.] There exist $C,M>0$ such that
			\begin{equation*}
				\left|1-e^{\pm 2\pi i (c_0\mu_r(\eta)-iq)}\right|\geq C\jp{\eta}^{-M},
			\end{equation*}
			for all $[\eta]\in \widehat{G}$, $1 \leq r \leq d_\eta$, whenever $c_0\mu_r(\eta)-iq\notin \Z$.
		\end{description}
	\end{lemma} 

	\section{Necessary Conditions}\label{Sec-4}
	
	In this section, we will prove the necessary conditions stated in Theorem \ref{thmbsum} about the global hypoellipticity of the operator $L$ \eqref{op-L}. First, let us prove that the global hypoellipticity of the corresponding constant coefficient operator $L_0$ is a necessary condition.
	
	\begin{prop}\label{implicationGH}
		If $L$ is globally hypoelliptic, then $L_{0}$ is globally hypoelliptic.
	\end{prop}
	\begin{proof}
		Assume that $L_{0}$ is not globally hypoelliptic. By Theorem \ref{GH-KMR} we have two cases to consider:
		\begin{enumerate}[(i)]
			\item The set \begin{equation*}
				\mathcal{N}=\{(k,[\eta]) \in \Z\times \widehat{G}; k+c_0\mu_{r}(\eta)-iq=0, \textrm{ for some } 1 \leq r \leq d_\eta\}
			\end{equation*}
			has infinitely many elements or;
			\item for all $M>0$, there exists $k_M\in \Z$ and $[\eta_M] \in \widehat{G}$ satisfying
			$$
			0< |k_M+c_0\mu_r(\eta_M)-iq| \leq (|k|+\jp{\eta_M})^{-M},
			$$
			for some $1 \leq r \leq d_{\eta_M}$.
		\end{enumerate}
		
		\noindent {\bf Case (i):} Assume that there exists a sequence $[\eta_k] \in \widehat{G}$ such that $c_0\mu_{r}(\eta_k)-iq \in \Z$, for some $1 \leq r \leq d_{\eta_k}$. for all $k\in\N$. We may assume without loss of generality that $r=1$ for all $[\eta_k]\in \widehat{G}$. For each $k\in\N$, let $t_k \in [0,2\pi]$ such that
		$$
		m_k:=\max_{t\in[0,2\pi]} \int_{0}^t (\operatorname{Re}{q}-\mu_1(\eta_k)b(s)) \, \mathrm{d}s =  \int_{0}^{t_k} (\operatorname{Re}{q}-\mu_1(\eta_k)b(s)) \, \mathrm{d}s.
		$$
		Set
		$$
		\widehat{u}(t,\eta)_{rs} = \left\{
		\begin{array}{ll}
			e^{m_k}\exp\left\{-\int_0^t(i\mu_1(\eta_k)c(s)+q)\,\mathrm{d}s\right\}, &  \mbox{if $[\eta]=[\eta_k]$ and  $r=s=1$,} \\
			0, & \mbox{otherwise. }
		\end{array}
		\right.
		$$
		Since $c_0\mu_{1}(\eta_k)-iq \in \Z$, for all $k\in\N$, the sequence of functions $\{\widehat{u}(t,\eta)_{rs}\}$ is well-defined on $\mathbb{T}^1$. Notice that
		$$
		|\widehat{u}(t,\eta_k)_{11}| = \left| e^{m_k} \exp\left\{ - \int_0^t (\operatorname{Re}{q}-\mu_1(\eta_k)b(s)) \, \mathrm{d}s\right\}\right| \leq 1,
		$$
		by the definition of $m_k$. By Theorem \ref{caracdist}, we have that $u \in \mathcal{D}'(\mathbb{T}^1\times G)$. Moreover, we have
		$$
		|\widehat{u}(t_k,\eta_k)_{11}| = 1, 
		$$
		for all $k\in\N$. By Theorem \ref{caracsmooth} we conclude $u \notin C^\infty(\mathbb{T}^1\times G)$. Since each element of the sequence $\{\widehat{u}(t,\eta)_{rs}\}$ satisfies
		$$
		\partial_t\{\widehat{u}(t,\eta)_{rs}\}+i(\mu_r(\eta)c(t)-iq)\{\widehat{u}(t,\eta)_{rs}\}=0,
		$$
		for all $[\eta]\in\widehat{G}$, $1 \leq r,s\leq d_\eta$, we conclude that $L u=0$, which implies that $L$ is not globally hypoelliptic.
		
		\noindent {\bf Case (ii):} By the equivalence given in Lemma \ref{expestimate}, we can construct a sequence $[\eta_k]$ satisfying for all $k\in\N$
		\begin{equation}\label{estexp}
			0 < |1-e^{-2\pi i(c_0\mu_r(\eta_k)-iq)}| < \jp{\eta_k}^{-k},
		\end{equation}
		for some $1 \leq r \leq d_{\eta_k}$. We may assume $r=1$ for convenience of notation and $c_0\mu_1(\eta_k)-iq \notin \Z$ for all $k \in \Z$, because $\mathcal{N}$ is finite.
		
		For each $k\in\N$, choose $t_k\in[0,2\pi]$ such that
		$$
		\max_{t\in[0,2\pi]} \int_0^t (\mu_1(\eta_k)b(s)-\operatorname{Re}{q}) \, \mathrm{d}s =  \int_0^{t_k} (\mu_1(\eta_k)b(s)-\operatorname{Re}{q}) \, \mathrm{d}s.
		$$
		Notice that with this choice we have
		\begin{equation}\label{t_k}
			\int_{t_k}^t (\mu_1(\eta_k)b(s)-\operatorname{Re}{q}) \, \mathrm{d}s \leq 0, \quad \mbox{for all } t \in [0,2\pi].
		\end{equation}
		By the compactness of the torus, we may assume, by passing to a subsequence, that there exists $t_0 \in [0,2\pi]$ such that $t_k \to t_0$, as $k \to \infty$.
		
		Let $\varphi \in C^\infty(\mathbb{T}^1)$ be a real-valued smooth function satisfying $\mbox{supp} (\varphi) \subseteq I$, $0 \leq \varphi(t) \leq 1$, and $\int_0^{2\pi} \varphi(s)\, \mathrm{d}s >0$, where $I$ is a closed interval in $(0,2\pi)$ such that $t_0 \notin I$. 
		
		Consider
		$$
		\widehat{f}(t,\eta)_{rs}= \left\{\begin{array}{ll}(1-e^{-2\pi i(c_0\mu_1(\eta_k)-iq)}) \exp\left\{-i\int_{t_k}^t (\mu_1(\eta_k)c(w)-iq)\, \mathrm{d}w \right\} \varphi(t), &\mbox{if $[\eta]=[\eta_k]$, $r=s=1$}, \\
			0, & \mbox{otherwise,}
		\end{array}
		\right.
		$$
		for $t\in[0,2\pi]$. Since $\mbox{supp}(\varphi) \subseteq I$, the sequence $\{\widehat{f}(t,\eta)_{rs}\}$ is well-defined on $\mathbb{T}^1$. Let us show that $\{\widehat{f}(t,\eta)_{rs}\}$  defines a smooth function on $\mathbb{T}^1\times G$. For $\alpha \in \N$ we have
		{\small$$
			\left|\partial_t^\alpha \widehat{f}(t,\eta_k)_{11}\right|\!=\!\left| 1-e^{-2\pi i(c_0\mu_1(\eta_k)-iq)}\right|\! \left| \sum_{ \beta \leq \alpha} \binom{\alpha}{\beta} \partial^\beta_ t\exp\left\{-i\int_{t_k}^t (\mu_1(\eta_k)c(w)-iq)\, \mathrm{d}w \right\}  \partial^{\alpha-\beta}\varphi(t)  \right|.
			$$}
		By Fa\`a di Bruno's formula we have
		\begin{align*}
			\partial^\beta_ t\exp\left\{-i\int_{t_k}^t (\mu_1(\eta_k)c(w)-iq)\, \mathrm{d}w \right\} &= \sum_{\gamma \in \Delta(\beta)} \frac{\beta! }{\gamma!}\exp\left\{-i\int_{t_k}^t (\mu_1(\eta_k)c(w)-iq)\, \mathrm{d}w \right\} \\ & \quad  \times \prod_{j=1}^{\beta} \left(\frac{-i\partial_t^j\int_{t_k}^t (\mu_1(\eta_k)c(w)-iq)\, \mathrm{d}w}{j!}\right)^{\gamma_j},
		\end{align*}
		where $\Delta(\beta) = \bigg\{\gamma \in \N_0^{\beta}; \ \sum\limits_{j=1}^\beta j\gamma_j= \beta \bigg\}$. 
		
		Since for all $k\in\N$ we have $|\mu_1(\eta_k)| \leq \jp{\eta_k}$ and $
		\int_{t_k}^t (\operatorname{Re}{q}-\mu_1(\eta_k)b(w)) \, \mathrm{d}w \leq 0$, for all $t \in [0,2\pi]$, we obtain
		$$
		\partial^\beta_ t\exp\left\{-i\int_{t_k}^t (\mu_1(\eta_k)c(w)-iq)\, \mathrm{d}w \right\} \leq C_\beta \jp{\eta_k}^\beta,
		$$
		for some $C_\beta >0$. By \eqref{estexp} we obtain
		$$
		\left|\partial_t^\alpha \widehat{f}(t,\eta_k)_{11}\right|\leq C_\alpha \jp{\eta_k}^{\alpha-k},
		$$
		for some $C_\alpha >0$. By Theorem \ref{caracsmooth} we conclude that $f\in C^\infty(\mathbb{T}^1\times G)$. 
		
		Let us construct now a distribution $u\in \mathcal{D}'(\mathbb{T}^1\times G)\setminus C^\infty(\mathbb{T}^1\times G)$ satisfying $Lu=f$.
		By \eqref{solution2}, set
		$$
		\widehat{u}(t,\eta_k)_{11} = \frac{1}{1-e^{-2\pi i(c_0\mu_1(\eta_k)-iq)}}\int_0^{2\pi}\exp\left\{-i\int_{t-\tau}^t(c(w)\mu_1(\eta_k)-iq)\, \mathrm{d}w \right\} \widehat{f}(t-\tau,\eta_k)_{11} \, \mathrm{d}\tau
		$$
		and $\widehat{u}(t,\eta)_{rs}=0$ for all the other cases. 	Notice that for $t-\tau <0$ we need to use the $2\pi$--periodic extension of $f$ on de definition of $\widehat{u}(t,\eta_k)_{11}$. Hence,
		 \begin{align}
				\widehat{u}(t,\eta_k)_{11} &= \frac{1}{1-e^{-2\pi i(c_0\mu_1(\eta_k)-iq)}}\int_0^{2\pi}\exp\left\{-i\int_{t-\tau}^t(c(w)\mu_1(\eta_k)-iq)\, \mathrm{d}w \right\} \widehat{f}(t-\tau,\eta_k)_{11} \, \mathrm{d}\tau \nonumber \\
				 &= \int_0^{t} \exp\left\{ -i\int_{t-\tau}^t(c(w)\mu_1(\eta_k)-iq)\, \mathrm{d}w-i\int_{t_k}^{t-\tau} (\mu_1(\eta_k)c(w)-iq)\, \mathrm{d}w \right\} \varphi(t-\tau)\, \mathrm{d}\tau\nonumber \\
				  &+ \int_t^{2\pi} \exp\left\{ -i\int_{t-\tau}^t(c(w)\mu_1(\eta_k)-iq)\, \mathrm{d}w-i\int_{t_k}^{t-\tau+2\pi} (\mu_1(\eta_k)c(w)-iq)\, \mathrm{d}w \right\} \varphi(t-\tau+2\pi)\, \mathrm{d}\tau\nonumber\\
				  &=\int_0^{t} \exp\left\{ -i\int_{t_k}^t(c(w)\mu_1(\eta_k)-iq)\, \mathrm{d}w\right\} \varphi(t-\tau)\, \mathrm{d}\tau\label{solution} \\
				  &+ \int_t^{2\pi} \exp\left\{ -i\int_{t_k}^t(c(w)\mu_1(\eta_k)-iq)\, \mathrm{d}w-i2\pi(\mu_1(\eta_k)c_0-iq)\right\} \varphi(t-\tau+2\pi)\, \mathrm{d}\tau \nonumber
			\end{align}
		We have $0\leq \varphi(t)\leq1$, so by \eqref{t_k} we obtain
		\begin{align*}
			|\widehat{u}(t,\eta_k)_{11}| &=\left|\int_0^{t} \exp\left\{ -i\int_{t_k}^t(c(w)\mu_1(\eta_k)-iq)\, \mathrm{d}w\right\} \varphi(t-\tau)\, \mathrm{d}\tau\right.\nonumber \\
			&+ \int_t^{2\pi} \left.\exp\left\{ -i\int_{t_k}^t(c(w)\mu_1(\eta_k)-iq)\, \mathrm{d}w-i2\pi(\mu_1(\eta_k)c_0-iq)\right\} \varphi(t-\tau+2\pi)\, \mathrm{d}\tau \right|\nonumber\\
			&\leq \int_0^{t} \exp\left\{	\int_{t_k}^t (\mu_1(\eta_k)b(s)-\operatorname{Re}{q})\right\} |\varphi(t-\tau)|\, \mathrm{d}\tau\nonumber \\
			&+ \int_t^{2\pi} \exp\left\{\int_{t_k}^t (\mu_1(\eta_k)b(s)+\operatorname{Re}{q})-2\pi (\operatorname{Re}{q}-b_0\mu_1(\eta_k))\right\} |\varphi(t-\tau+2\pi)|\, \mathrm{d}\tau\nonumber\\
			&\leq t +(2\pi-t) \exp\left\{2\pi (\operatorname{Re}{q}-b_0\mu_1(\eta_k)) \right\} \leq 4\pi
		\end{align*}
		for sufficiently large $k$, where the last inequality comes from the fact that by the relation  \eqref{estexp} we have $|e^{-2\pi i(c_0\mu_1(\eta_k)-iq)}| \to 1$, when $k \to\infty$.
		By Theorem \ref{caracdist}, we have $u \in \mathcal{D}'(\mathbb{T}^1\times G)$.
		
		Notice that if $t_0 > \sup I$, then $t_k > \sup I$, for $k$ sufficiently large, which implies that $t_k \geq \tau$, for every $\tau \in \textrm{supp}(\varphi)$.  By \eqref{solution} we obtain
		$$
		|\widehat{u}(t_k,\eta_k)_{11}| = \int_{0}^{2\pi} \varphi(t_k-\tau) \, \mathrm{d}\tau = \|\varphi\|_{L^1(\mathbb{T}^1)} >0.
		$$
		On the other hand, if $t_0 < \inf I$, we have for $k$ sufficiently large. that $t_k < \tau$, for every $\tau \in \textrm{supp}(\varphi)$. By \eqref{solution} we have
		$$
		|\widehat{u}(t_k,\eta_k)_{11}| = \exp\left\{2\pi (\operatorname{Re}{q}-b_0\mu_1(\eta_k)) \right\}  \int_{0}^{2\pi} \varphi(t_k-\tau+2\pi) \, \mathrm{d}\tau > \tfrac{1}{2}\|\varphi\|_{L^1(\mathbb{T}^1)} >0.
		$$
		By Theorem \ref{caracsmooth} we conclude that $u\notin C^\infty(\mathbb{T}^1\times G)$. Therefore $L$ is not globally hypoelliptic.
	\end{proof}

	\begin{thm} \label{nec} Assume that $b\not\equiv0$. If $L=\partial_t+(a_0+ib(t))X+q$ is globally hypoelliptic then $b$ does not change sign.
	\end{thm} 
	\begin{proof}
		Suppose that $b$ change sign and $b_0 > 0$. Consider
		$$
		G(t,\tau) = \int_{t}^{t+\tau} (a_0+ib(w)) \, \mathrm{d}w = a_0\tau + i\int_t^{t+\tau} b(w) \mathrm{d}w, \quad t,\tau \in [0,2\pi]
		$$
		and define
		$$
		B = \min_{0\leq t,\tau \leq 2\pi} \operatorname{Im}{G(t,\tau)} = \operatorname{Im}{G(t_0,\tau_0)} = \int^{t_0+\tau_0}_{t_0} b(w) \, \mathrm{d}w.
		$$
		Since $b$ change sign, we have $B<0$. Moreover, we can consider $t_0, \tau_0 \in (0,2\pi)$ and $b(0)\neq 0$. It can be shown that $b(t_0+\tau_0) = 0$, which implies that $t_0+\tau_0 \in (0,2\pi)$.
		
		Let $\varphi \in C^\infty(\mathbb{T}^1)$ such that $\text{supp}(\varphi) \subset [t_0+\tau_0-\delta,t_0+\tau_0+\delta] \subset (0,t_0)$ with $\varphi(t)\equiv 1$ for $t \in [t_0+\tau_0-\delta/2, t_0+\tau_0+\delta/2]$ and $0 \leq \varphi(t) \leq 1 $.
		
		Let us construct a distribution $u \in \mathcal{D}'(\mathbb{T}^1\times G) \backslash C^\infty(\mathbb{T}^1\times G)$ such that $Lu=f\in C^\infty(\mathbb{T}^1\times G)$. From \eqref{growth}, there exist $C>0$ and a sequence $\{[\eta_j]\}_{j\in \N}$ in $\widehat{G}$ such that 
		\begin{equation}\label{grow}
			C\langle \eta_j \rangle^{\frac{1}{2}} \leq \mu_{d_{\eta_j}}(\eta_j),
		\end{equation} 
		for all $j\in \N$. By Proposition \ref{implicationGH} we have that $L_{0}$ is globally hypoelliptic. In particular, the set
		$$
		\mathcal{N}=\{[\eta]\in \widehat{G}; \mu_r(\eta)c_0-iq \in \Z, \mbox{ for some } 1\leq r\leq d_\eta\}
		$$
		is finite and we may assume that $[\eta_j]\notin \mathcal{N}$, for all $j \in \N$. Define
		$$\widehat{f}(t,\eta)_{rs} \!=\! \left\{
		\begin{array}{ll} \tilde{c}_ke^{B\mu_r(\eta_j)}\varphi(t) e^{-i\mu_r(\eta_j)a_0(t-t_0)}, & \mbox{if } [\eta]=[\eta_j], \mbox{ for some $j\in \N$, and } r=s=d_{\eta_j}, \\
			0, & \mbox{otherwise },
		\end{array}
		\right.
		$$
		where $\tilde{c}_k:=e^{2\pi i(\mu_r(\eta_j)c_0-iq)}-1$. In order to prove that the sequence $\{ \widehat{f}(t,\eta)_{rs}\}$ defines a smooth function in $\mathbb{T}^1\times G$, it is enough to consider the representations $[\eta_j]$ and the components $r=s=d_{\eta_j}$. 
		
		Notice that
		\begin{align*}
			|\partial^\alpha_t \widehat{f}(t,{\eta_j})_{d_{\eta_j} d_{\eta_j} }| \!&= \!\left|(e^{2\pi i(\mu_{d_{\eta_j}}({\eta_j})c_0-iq)}-1) e^{B\mu_{d_{\eta_j}}({\eta_j})}\sum_{\beta \leq \alpha} \binom{\alpha}{\beta} \partial_t^\beta e^{-i\mu_{d_{\eta_j}}({\eta_j})a_0(t-t_0)}\partial_t^{\alpha-\beta}\varphi(t)\right| \nonumber \\ 
			&\leq \!\left|e^{2\pi i(\mu_{d_{\eta_j}}({\eta_j})c_0-iq)}-1\right| e^{B\mu_{d_{\eta_j}}({\eta_j})}\sum_{\beta \leq \alpha} \binom{\alpha}{\beta} \left|\partial_t^\beta e^{-i\mu_{d_{\eta_j}}({\eta_j})a_0(t-t_0)}\right|\!\left|\partial_t^{\alpha-\beta}\varphi(t)\right|\!. \nonumber
		\end{align*}
		
		Observe that 
		$$
		\left|e^{2\pi i(\mu_{d_{\eta_j}}({\eta_j})c_0-iq)}-1\right| \leq \left|e^{2\pi i(\mu_{d_{\eta_j}}({\eta_j})c_0-iq)}\right| +1 \leq e^{2\pi(-\mu_{d_{\eta_j}}({\eta_j})b_0+\operatorname{Re}{q})} + 1 \leq C,
		$$
		for some $C>0$, because $b_0\geq0$ and $\mu_{d_{\eta_j}}({\eta_j}) \to \infty$. Notice that
		$$
		\left|\partial_t^\beta e^{-i\mu_{d_{\eta_j}}({\eta_j})a_0(t-t_0)}\right| = \left|(-i\mu_{d_{\eta_j}}({\eta_j})a_0)^\beta e^{-i\mu_{d_{\eta_j}}({\eta_j})a_0(t-t_0)}\right| \leq C_\beta\jp{{\eta_j}}^\beta.
		$$
		Moreover, since $B<0$, we obtain from \eqref{grow} $$e^{B\mu_{d_{\eta_j}}({\eta_j})} \leq e^{CB\jp{{\eta_j}}^{\frac{1}{2}}}.$$ 
		Hence,
		$$
		|\partial^\alpha_t \widehat{f}(t,{\eta_j})_{d_{\eta_j} d_{\eta_j} }| \leq C_\alpha e^{CB\jp{{\eta_j}}^{\frac{1}{2}}} \jp{{\eta_j}}^\alpha.
		$$
		Since $B<0$, for any $N>0$ there exists $C_{\alpha N}$ such that
		$$
		|\partial^\alpha_t \widehat{f}(t,\eta)_{d_\eta d_\eta }| \leq C_{\alpha N} \jp{\eta}^{-N}.$$
		By Theorem \ref{carac}, we conclude that $f \in C^\infty(\mathbb{T}^1\times G)$. If $Lu=f$, then
		\begin{equation*}\label{fouriereq2}
			\widehat{f}(t,\eta)_{rs}=\widehat{Lu}(t,\eta)_{rs} = \partial_t \widehat{u}(t,\eta)_{rs} + i\mu_r(\eta)(a_0+ib(t)-iq)\widehat{u}(t,\eta)_{rs},
		\end{equation*}
		for $1\leq r,s, \leq d_\eta$. Since $\mu_{d_{\eta_j}}(\eta_j)c_0-iq \notin \Z$, by \eqref{solutionplus} we obtain
		\begin{align*}
			\widehat{u}(t,{\eta_j})_{d_{\eta_j} d_{\eta_j}} &= \frac{1}{e^{2\pi i(\mu_{d_{\eta_j}}({\eta_j})c_0-iq)}-1} \int_{0}^{2\pi} e^{q\tau}e^{i\mu_{d_{\eta_j}}({\eta_j})G(t,\tau)} \widehat{f}(t+\tau,{\eta_j})_{rs} \, \mathrm{d}\tau \nonumber \\
			&= e^{-i\mu_{d_{\eta_j}}({\eta_j})a_0(t-t_0)} \int_{0}^{2\pi}e^{q\tau}e^{\mu_{d_{\eta_j}}({\eta_j})(B-\operatorname{Im}{G(t,\tau)})} \varphi(t+\tau) \, \mathrm{d}\tau.
		\end{align*}
		In all the other cases set $\widehat{u}(t,\eta)_{rs}=0$. First, let us show that the sequence $\{ \widehat{u}(t,\eta)_{rs}\}$  defines a distribution $u\in\mathcal{D}'(\mathbb{T}^1\times G)$, where
		$$
		u=\sum_{[\eta]\in \widehat{G}} d_\eta \sum_{r,s=1}^{d_\eta} \widehat{u}(t,\eta)_{rs}{\eta_{sr}}.
		$$ 
		
		In order to apply the Theorem \ref{caracdist}, it is enough to consider the case where $[\eta]=[\eta_j]$ for some $j\in \N$ and $r=s=d_{\eta_j}$ because the other cases are well-controlled. 
		
		Let $\psi \in C^\infty(\mathbb{T}^1)$, then 
		\begin{align*}
			\left|\left( \widehat{u}(t,{\eta_j})_{rs},\psi \right) \right| &= \left| \int_{0}^{2\pi}e^{-i\mu_{d_{\eta_j}}({\eta_j})a_0(t-t_0)} \int_{0}^{2\pi}e^{q\tau}e^{\mu_{d_{\eta_j}}({\eta_j})(B-\operatorname{Im}{G(t,\tau)})} \varphi(t+\tau) \, \mathrm{d}\tau \,\psi(t)\, dt\right| \nonumber \\
			&\leq \int_{0}^{2\pi}\int_{0}^{2\pi} e^{\operatorname{Re}{q} \tau}e^{\mu_{d_{\eta_j}}({\eta_j})(B-\operatorname{Im}{G(t,\tau)})} |\varphi(t+\tau)|\,|\psi(t)|\, \mathrm{d}\tau dt \nonumber\\
			&\leq (2\pi)^2 \|\varphi\|_{\infty} \|\psi\|_{\infty} \nonumber \\
			& \leq  K p_1(\psi) \langle {\eta_j} \rangle. \nonumber
		\end{align*}
		Notice that here we have used the fact that $\mu_{d_{\eta_j}}({\eta_j})(B-\operatorname{Im}{G(t,\tau)}) \leq 1$. Therefore $u \in \mathcal{D}'(\mathbb{T}^1\times G)$. Consider the function
		$$
		\theta(\tau)=B-\operatorname{Im}{G(t_0,\tau)} = B- \int_{t_0}^{t_0+\tau} b(w) \, \mathrm{d}w.
		$$
		
		We may consider $\delta$ small enough in the properties of $\varphi$  such that either $\cos(\operatorname{Im}{q})$ or $\sin(\operatorname{Im}{q}
		)$ does not change sign on $(\tau_0+\delta,\tau_0+\delta)$. Assume without loss of generality that $\sin(\operatorname{Im}{q})\geq0$ on $(\tau_0+\delta,\tau_0+\delta).$ Thus
		\begin{align*}
			|\widehat{u}(t_0,{\eta_j})_{d_{\eta_j} d_{\eta_j}} | &= \left|\int_{0}^{2\pi} e^{q\tau}e^{\mu_{d_{\eta_j}}({\eta_j}) \theta(\tau)} \varphi(t_0+\tau) \, \mathrm{d}\tau \right| \nonumber \\
			&\geq \int_{\tau_0-\delta}^{\tau_0+\delta} e^{\operatorname{Re}{q}\tau}\sin(\operatorname{Im}{q}\tau)e^{\mu_{d_{\eta_j}}({\eta_j}) \theta(\tau)} \varphi(t_0+\tau) \, \mathrm{d}\tau \nonumber \\
			&\geq \int_{\tau_0-\delta/2}^{\tau_0+\delta/2} e^{\operatorname{Re}{q}\tau}\sin(\operatorname{Im}{q}\tau)e^{\mu_{d_{\eta_j}}({\eta_j}) \theta(\tau)} \, \mathrm{d}\tau \nonumber\\
			& \geq  K\int_{\tau_0-\delta/2}^{\tau_0+\delta/2} e^{\langle {\eta_j} \rangle \theta(\tau)} \, \mathrm{d}\tau \nonumber,
		\end{align*}
		where we use the fact that $\theta(\tau) \leq 0$, for all $\tau \in [0,2\pi]$, $\mu_{d_{\eta_j}}({\eta_j}) \leq \langle {\eta_j} \rangle$ for all $[{\eta_j}] \in \widehat{G}$, and there exists $K>0$ such that $e^{\operatorname{Re}{q}\tau}\sin(\operatorname{Im}{q}\tau)\geq K$ on $[\tau_0-\delta/2,\tau_0+\delta/2]$.
		
		Let us analyze the behavior of the function   
		$$
		J({\eta_j})= \int_{\tau_0-\delta/2}^{\tau_0+\delta/2} e^{\langle {\eta_j} \rangle \theta(\tau)} \, \mathrm{d}\tau
		$$
		when $\langle {\eta_j} \rangle \to \infty$.
		We have
		$$
		\theta(\tau_0) =  B - \int_{t_0}^{t_0+\tau_0} b(w) \, \mathrm{d}w = B-B = 0
		$$
		and 
		$$
		\theta'(\tau_0) = -b(t_0+\tau_0) = 0.
		$$
		Thus by Taylor's formula, we have
		$$
		\theta(\tau_0 + h) = \theta(\tau_0) + \theta'(\tau_0)h + \frac{1}{2} \theta''(\tau_0+\theta(h))h^2=\frac{1}{2} \theta''(\tau_0+\theta(h))h^2,
		$$
		for $h \in (\tau_0-\delta/2,\tau_0+\delta/2)$ and $\theta(h) \in [\tau_0-\delta/2,\tau_0+\delta/2]$.
		Let
		$$
		M:= \sup_{\tau_0 -\delta \leq y \leq \tau_0+\delta} \left|\frac{\theta''(y)}{2}\right|.
		$$
		
		If $M=0$ then $\theta \equiv 0$ in $[\tau_0-\delta/2, \tau_0+\delta/2]$. Thus
		$$
		\int_{\tau_0-\delta/2}^{\tau_0+\delta/2} e^{\langle {\eta_j} \rangle \theta(\tau)} \, \mathrm{d}\tau = \int_{\tau_0-\delta/2}^{\tau_0+\delta/2} e^0 \, \mathrm{d}\tau = \delta \geq \frac{C_1}{\sqrt{\langle {\eta_j} \rangle}},
		$$
		for some $C_1>0$.
		
		If $M>0$, then
		$$
		-\theta(\tau_0+h) = -\frac{1}{2}\theta''(\tau_0+\theta(h))h^2 \leq Mh^2.
		$$
		So
		$$
		\langle {\eta_j} \rangle \theta(\tau_0+h) \geq -M\langle {\eta_j} \rangle h^2.
		$$
		Thereby
		$$
		\int_{\tau_0-\delta/2}^{\tau_0+\delta/2} e^{\langle {\eta_j} \rangle \theta(\tau)} \, \mathrm{d}\tau =  \int_{-\delta/2}^{\delta/2} e^{\langle {\eta_j} \rangle \theta(\tau_0+h)} \, dh \geq \int_{-\delta/2}^{\delta/2} e^{-M\langle {\eta_j} \rangle h^2} \, dh
		\geq\frac{C_2}{\sqrt{\langle {\eta_j} \rangle}},
		$$
		for some $C_2>0$. Considering $C=\max\{KC_1,KC_2 \}$, we have
		$$| \widehat{u}(t_0,{\eta_j})_{d_{\eta_j} d_{\eta_j}} | \geq \frac{C}{\sqrt{\langle {\eta_j} \rangle}},$$
		for all $[{\eta_j}]\in \widehat{G}$. Therefore $u \notin C^\infty(\mathbb{T}^1\times G)$. 
		
		The case where $b_0<0$ is analogous to the previous one, but needs some adaptions. Here we take
		$$
		\widetilde{B} := \max_{0\leq t,\tau \leq 2\pi} \operatorname{Im}{H(t,\tau)} = \operatorname{Im}{H(t_1,\tau_1)} = \int_{t_1-\tau_1}^{t_1} b(w) \, \mathrm{d}w.
		$$
		Since $b$ change sign, then $\widetilde{B}>0$. For $r=s=d_{\eta_j}$, define
		$$
		\widehat{f}(t,{\eta_j})_{rs} = (1-e^{-2\pi i(\mu_r({\eta_j})c_0-iq)}) e^{-\widetilde{B}\mu_r({\eta_j})}\widetilde{\varphi}(t) e^{-i\mu_r({\eta_j})a_0(t-t_1)},
		$$
		where $\widetilde{\varphi} \in C^{\infty}(\mathbb{T}^1)$ satisfies similar properties of $\varphi$. One can shows that $f \in C^\infty(\mathbb{T}^1\times G)$ and there exists $u \in \mathcal{D}'(\mathbb{T}^1\times G) \backslash C^\infty(\mathbb{T}^1\times G)$ such that $Lu=f$. For this, define for $r=s=d_{\eta_j}$
		$$
		\widehat{u}(t,{\eta_j})_{rs} = e^{-i\mu_r({\eta_j}) a_0(t-t_1)} \int_{0}^{2\pi} e^{\mu_r({\eta_j})(\operatorname{Im}{H(t,\tau)}-\widetilde{B})} \widetilde{\varphi}(t-\tau) \, \mathrm{d}\tau.
		$$
		The proof that $u \in \mathcal{D}'(\mathbb{T}^1\times G) \backslash C^\infty(\mathbb{T}^1\times G)$ is similar to the previous case and it will be omitted.
	\end{proof}

	\section{Sufficient Conditions}\label{Sec-5}
	
	In this section, we will prove the sufficient conditions on Theorem \ref{thmbsum} regarding the global hypoellipticity of $L$. Because of Proposition \ref{implicationGH}, from now we will assume that $L_{0}$ is global hypoelliptic. By Theorem \ref{GH-KMR}, this assumption implies that the set 
	\begin{equation*}\label{setNc4}
		\mathcal{N}=\{(k,[\eta]) \in \Z\times \widehat{G}; k+c_0\mu_{r}(\eta)-iq=0, \textrm{ for some } 1 \leq r \leq d_\eta\}
	\end{equation*}
	is finite and there exist $C,M>0$ such that
	\begin{equation*}\label{ADC1}
		|k+c_0\mu_r(\eta)-iq| \geq C(|k|+\jp{\eta})^{-M},
	\end{equation*}
	for all $k\in \Z$, $[\eta]\in \widehat{G}$, $1 \leq r \leq d_\eta$, whenever $k+c_0\mu_r(\eta)-iq\neq0$.
	
	\begin{thm}\label{suf}
		Assume that $L_{0}$ is globally hypoelliptic and $b\not\equiv 0$. If $b$ does not change sign then $L$ is globally hypoelliptic.
	\end{thm}
	\begin{proof}
		Assume that $b(t) \geq 0 $ for all $t \in \mathbb{T}^1$. By hypothesis, $b_0 \neq 0$, where $c_0=a_0+ib_0$. Notice that the global hypoellipticity of $L_{0}$ implies that $\mu_r(\eta)c_0-iq \in \Z$ for only finitely many representations. So there is no loss of generality to assume that $\mu_r(\eta)c_0-iq \notin \Z$. Let $f\in C^\infty(\mathbb{T}^1\times G)$ such that $L u = f$, for some $u \in \mathcal{D}'(\mathbb{T}^1\times G)$. Let us show that $u \in C^\infty(\mathbb{T}^1\times G)$. 
		
		Define
		$$
		H(\tau,t) = c_0\tau -C(t-\tau)+C(t).
		$$
		For $\mu_r(\eta)<0$, consider the solution \eqref{solution2}:
		\begin{equation}\label{solution3}
			\widehat{u}(t,\eta)_{rs}=\frac{1}{1-e^{-2\pi i (\mu_r(\eta)c_0-iq)}} \int_{0}^{2\pi}e^{-q\tau} e^{-i\mu_r(\eta)H(\tau,t)} \widehat{f}(t-\tau,\eta)_{rs}\mathrm{d}\tau
		\end{equation}
		and for $\mu_r(\eta)\geq0$, consider the solution \eqref{solutionplus}:
		\begin{equation}\label{solution4}
			\widehat{u}(t,\eta)_{rs}=\frac{1}{e^{2\pi i( \mu_r(\eta)c_0-iq)}-1} \int_{0}^{2\pi}e^{q\tau} e^{-i\mu_r(\eta)H(-\tau,t)} \widehat{f}(t+\tau,\eta)_{rs}\mathrm{d}\tau.
		\end{equation}
		
		Notice that
		\begin{align*}
			H(\tau,t)&= c_0\tau -C(t-\tau)+C(t) \nonumber \\
			& = c_0 \tau - \int_{0}^{t-\tau}c(w) \mathrm{d}w + (t-\tau)c_0 + \int_{0}^{t} c(w)\mathrm{d}w - tc_0 \nonumber \\
			& = \int_{t-\tau}^t c(w) \mathrm{d}w \nonumber
		\end{align*}
		So, using the fact that $b(t) \geq 0$, for all $t \in \mathbb{T}^1$, we obtain
		\begin{equation}\label{impab}
			\operatorname{Im}{H(\tau,t)} = \operatorname{Im}{\int_{t-\tau}^t c(w) \mathrm{d}w} = \int_{t-\tau}^t b(w) \mathrm{d}w \geq 0.
		\end{equation}
		\begin{equation*}
			\operatorname{Im}{H(-\tau,t)} = \operatorname{Im}{ \int_{t+\tau}^t c(w) \mathrm{d}w} = \int_{t+\tau}^t b(w) \mathrm{d}w \leq 0.
		\end{equation*} 
		
		Notice that there exist $K>0$ such that
		$$
		\left|e^{\pm q\tau}\right|  \leq K,
		$$
		because $0 \leq \tau \leq 2\pi$.  Let $\alpha \in \N_0$ and $\mu_r(\eta)c_0-iq \notin \Z$. If $\mu_r(\eta)<0$, by \eqref{solution3} we have
		{\small
			\begin{align*}
				\bigl| \partial_t^\alpha \widehat{u}(t,\eta)_{rs} \bigr| &= \bigg| \frac{1}{1-e^{-2\pi i (\mu_r(\eta)c_0-iq)}} \int_{0}^{2\pi} \partial_t^\alpha \bigl[e^{-q\tau}e^{-i\mu_r(\eta)H(\tau,t)} \widehat{f}(t-\tau,\eta)_{rs}\bigr] \mathrm{d}\tau \bigg| \nonumber  \\
				& \leq  \bigg| \frac{1}{1-e^{-2\pi i (\mu_r(\eta)c_0-iq)}} \bigg| \int_0^{2\pi} |e^{-q\tau}|  \sum_{\beta=0}^\alpha \binom{\alpha}{\beta} \left|\partial_t ^\beta e^{-i\mu_r(\eta)H(\tau,t)} \right|\left|\partial_t^{\alpha-\beta}\widehat{f}(t-\tau,\eta)_{rs}\right|\nonumber \mathrm{d}\tau. \nonumber
			\end{align*}
		}
		By the assumption of the global hypoellipticity of $L_{0}$, we obtain from Lemma \ref{expestimate} constants $C,M>0$ satisfying
		\begin{equation}\label{estimativaexp}
			|1-e^{-2\pi i (\mu_r(\eta)c_0-iq)}|^{-1} \leq C\jp{\eta}^{M}, 
		\end{equation}
		for all $[\eta]\in \widehat{G}$, $1 \leq r \leq d_\eta$, whenever $c_0\mu_r(\eta)-iq\notin \Z$.	By Fa\`a di Bruno's Formula, we have
		$$
		\partial_t^\beta e^{-i\mu_r(\eta)H(t,t)} = \sum_{\gamma \in \Delta(\beta)} \frac{\beta! }{\gamma!} (-i\mu_r(\eta))^{|\gamma|} e^{-i\mu_r(\eta)H(\tau,t)} \prod_{j=1}^{\beta} \left(\frac{\partial_t^j H(\tau,t)}{j!}\right)^{\gamma_j},
		$$
		where $\Delta(\beta) = \bigg\{\gamma \in \N_0^{\beta}; \ \sum\limits_{j=1}^\beta j\gamma_j= \beta \bigg\}$. Hence,
		$$
		\left|\partial_t^\beta e^{-i\mu_r(\eta)H(t,t)} \right| \leq \sum_{\gamma \in \Delta(\beta)} \frac{\beta! }{\gamma!} 
		|\mu_r(\eta)|^{|\gamma|} e^{\mu_r(\eta)\operatorname{Im}{H(\tau,t)}} \prod_{j=1}^{\beta} \left|\frac{\partial_t^j H(\tau,t)}{j!}\right|^{\gamma_j}.
		$$
		Notice that by \eqref{controleigen} we have
		$$
		|\mu_r(\eta)|^{|\gamma|} \leq \jp{\eta}^{|\gamma|}\leq \jp{\eta}^\beta,
		$$
		for all $[\eta]\in \widehat{G}$, $1 \leq r \leq d_\eta$, and $\gamma \in \Delta(\beta)$. Moreover, by \eqref{impab} we have
		$$
		e^{\mu_r(\eta)\operatorname{Im}{H(\tau,t)}} \leq 1.
		$$
		Thus,
		$$
		\left|\partial_t^\beta e^{-i\mu_r(\eta)H(t,t)} \right| \leq \jp{\eta}^{\beta}\sum_{\gamma \in \Delta(\beta)} \frac{\beta! }{\gamma!} \prod_{j=1}^{\beta} \left|\frac{\partial_t^j H(\tau,t)}{j!}\right|^{\gamma_j}.
		$$
		By the continuity of the function $H$ and the compactness of $\mathbb{T}^1$, for all $\beta \in \N_0$ there exists $C_\beta>0$ such that
		$$
		\sum_{\gamma \in \Delta(\beta)} \frac{\beta! }{\gamma!} \prod_{j=1}^{\beta} \left|\frac{\partial_t^j H(\tau,t)}{j!}\right|^{\gamma_j} \leq C_\beta,
		$$
		for all $0 \leq t,\tau \leq 2\pi$.
		
		Let $N>0$. Since $f \in C^\infty(\mathbb{T}^1\times G)$, by Theorem \ref{caracsmooth} for every $\beta\leq \alpha$ there exists ${C}_{\beta N}>0$ such that
		$$
		|\partial_t^{\alpha-\beta}\widehat{f}(t,\eta)_{rs}| \leq{C}_{\beta N} \langle \eta \rangle^{-(N+\beta+M)},
		$$ 
		for all $t \in \mathbb{T}^1$, with $M$ as in \eqref{estimativaexp}. Therefore,
		{\small
			\begin{align*}
				\bigl| \partial_t^\alpha \widehat{u}(t,\eta)_{rs} \bigr| &\leq   \bigg| \frac{1}{1-e^{-2\pi i (\mu_r(\eta)c_0-iq)}} \bigg| \int_0^{2\pi} |e^{-q\tau}|  \sum_{\beta=0}^\alpha \binom{\alpha}{\beta} \left|\partial_t ^\beta e^{-i\mu_r(\eta)H(\tau,t)} \right|\left|\partial_t^{\alpha-\beta}\widehat{f}(t-\tau,\eta)_{rs}\right|\nonumber \mathrm{d}\tau.\\
				&\leq   KC\jp{\eta}^M  \int_0^{2\pi} \sum_{\beta=0}^\alpha \binom{\alpha}{\beta}C_\beta \jp{\eta}^\beta C_{\beta N} \jp{\eta}^{-(N+\beta+M)}\, \mathrm{d}\tau\\
				& \leq C_{\alpha N} \jp{\eta}^{-N}.
			\end{align*}
		}
		We can obtain the same type of estimate when $\mu_r(\eta)\geq 0$. In this case, it is enough to consider the expression \eqref{solution4} to take the derivatives. We can adjust $C_{\alpha N}$, if necessary, to obtain
		$$
		\bigl| \partial_t^\alpha \widehat{u}(t,\eta)_{rs} \bigr|  \leq C_{\alpha N} \langle \eta\rangle^{-N},
		$$
		for every $[\eta] \in \widehat{G}$, $1\leq r,s\leq d_\eta$. By Theorem \ref{caracsmooth} we conclude that $u \in C^\infty(\mathbb{T}^1\times G)$. The case  $b(t) \leq 0$, for all $t\in \mathbb{T}^1$, is totally analogue, just use \eqref{solution3} for $\mu_r(\eta)\geq 0$ and \eqref{solution4} for $\mu_r(\eta)<0$.
	\end{proof}

	\begin{example}\label{exe}
		Let $G$ be a compact Lie group and $q\in i(\R\setminus\Z)$. The operator
		$$
		L:=\partial_t+(e^{it}+i)X+q
		$$
		is globally hypoelliptic. Indeed, we have $\operatorname{Im}{e^{it}+i}=\sin(t)+1\not\equiv 0$ and the operator
		$L_{0}=\partial_t+iX+q$ is globally hypoelliptic by Theorem \ref{GH-KMR} because in this case we have
		$$
		\mathcal{N}=\{[\eta]\in\widehat{G}; i\mu_r(\eta)-iq\in\Z\}=\varnothing,
		$$
		and
		$$
		|k+i\mu_r(\eta)-iq| \geq |k-iq| \geq C,
		$$
		for some $C>0$, for all $k\in\Z$, $[\eta]\in \widehat{G}$, $1 \leq r \leq d_\eta$. Since $\operatorname{Im}{e^{it}+i}=\sin(t)+1$ does not change sign, by Theorem \ref{suf} we conclude that $L$ is globally hypoelliptic.
	On the other hand, the operator 
	$$
		L:=\partial_t+(2e^{it}+i)X+q.
		$$
		is not globally hypoelliptic. Indeed, notice that $\operatorname{Im}{2e^{it}+i}=2\sin(t)+1\not\equiv 0$ and the operator $L_{0}=\partial_t+iX+q=L_0$ is globally hypoelliptic and  since  $\operatorname{Im}{2e^{it}+i}=2\sin(t)+1$ changes sign, we conclude by Theorem \ref{nec} that $L$ is not globally hypoelliptic.
		
		Finally, let us see that $b$ does not change sign is insufficient for the global hypoellipticity of $L$. Consider $G=\mathbb{S}^3$ and $X=i\partial_0$ (see \cite{KMP21} and \cite{livropseudo} for the precise Fourier analysis on $\mathbb{S}^3$) . We may identify $\widehat{\mathbb{S}^3} \sim \frac{1}{2}\mathbb{N}$ and in this case we have that $\sigma_X(\ell)_{rs} = ir\delta_{rs}$, for all $\ell \in \frac{1}{2}\mathbb{N}$, $r-\ell,s-\ell \in \mathbb{Z}$. The operator 
		$$
		L=\partial_t+(e^{it}+i)X+in,
		$$
		where $n\in \mathbb{Z}$, is not globally hypoelliptic. Indeed, we have the corresponding constant coefficient operator $L_0$ is given by $L_0=\partial_t+iX+in$ and
		$$
		\mathcal{N}=\left\{\ell \in \frac{1}{2}\mathbb{N}; -r+n \in \mathbb{Z}\right\} = \mathbb{N},
		$$
		which fails the condition 1. of Theorem \ref{GH-KMR} for the global hypoellipticity of $L_0$. Hence, by Proposition \ref{implicationGH} we conclude that $L$ is not globally hypoelliptic. We point out that this phenomenon occurs because when $G$ is not a torus, the fact that $b$ does not change sign does not imply the global hypoellipticity of $L_0$.
	\end{example}

	\section*{Acknowledgments}
	
	This study was financed by the National Council for Scientific and Technological Development – CNPq [423458/2021-3].
	\bibliographystyle{abbrv}
	\bibliography{biblio}

\begin{thebibliography}{10}

\bibitem{Ber94}
A.~P. Bergamasco.
\newblock Perturbations of globally hypoelliptic operators.
\newblock {\em J. Differential Equations}, 114(2):513--526, 1994.

\bibitem{CKKM24}
D.~Cardona, A.~Kirilov, A.~P. Kowacs, and W.~A.~A. de~Moraes.
\newblock On the {S}obolev boundedness of vector fields on compact {R}iemannian
  manifolds.
\newblock {\em J. Pseudo-Differ. Oper. Appl.}, 16(3):52, 2025.

\bibitem{AGKM18}
F.~de~\'{A}vila Silva, R.~B. Gonzalez, A.~Kirilov, and C.~de~Medeira.
\newblock Global hypoellipticity for a class of pseudo-differential operators
  on the torus.
\newblock {\em J. Fourier Anal. Appl.}, 25(4):1717--1758, 2019.

\bibitem{FR16}
V.~Fischer and M.~Ruzhansky.
\newblock {\em Quantization on nilpotent {L}ie groups}, volume 314 of {\em
  Progress in Mathematics}.
\newblock Birkh\"{a}user/Springer, [Cham], 2016.

\bibitem{GW72}
S.~J. Greenfield and N.~R. Wallach.
\newblock Global hypoellipticity and {L}iouville numbers.
\newblock {\em Proc. Amer. Math. Soc.}, 31:112--114, 1972.

\bibitem{GW73a}
S.~J. Greenfield and N.~R. Wallach.
\newblock Globally hypoelliptic vector fields.
\newblock {\em Topology}, 12:247--254, 1973.

\bibitem{Hou79}
J.~Hounie.
\newblock Globally hypoelliptic and globally solvable first-order evolution
  equations.
\newblock {\em Trans. Amer. Math. Soc.}, 252:233--248, 1979.

\bibitem{KMP21}
A.~Kirilov, W.~A.~A. de~Moraes, and R.~Paleari.
\newblock Global analytic hypoellipticity for a class of evolution operators on
  {$\mathbb{T}^1\times \mathbb{S}^3$}.
\newblock {\em J. Differential Equations}, 296:699--723, 2021.

\bibitem{KMR19}
A.~Kirilov, W.~A.~A. de~Moraes, and M.~Ruzhansky.
\newblock Partial {F}ourier series on compact {L}ie groups.
\newblock {\em Bull. Sci. Math.}, 160:102853, 27, 2020.

\bibitem{KMR19b}
A.~Kirilov, W.~A.~A. de~Moraes, and M.~Ruzhansky.
\newblock Global hypoellipticity and global solvability for vector fields on
  compact {L}ie groups.
\newblock {\em J. Funct. Anal.}, 280(2):108806, 2021.

\bibitem{livropseudo}
M.~Ruzhansky and V.~Turunen.
\newblock {\em Pseudo-differential operators and symmetries}, volume~2 of {\em
  Pseudo-Differential Operators. Theory and Applications}.
\newblock Birkh\"{a}user Verlag, Basel, 2010.
\newblock Background analysis and advanced topics.

\end{thebibliography}
\end{document}